\documentclass{article}

\usepackage{
	amsmath, 
	amsthm, 
	amssymb,
	fancyhdr, 
	setspace,
	tikz-cd,
}
\usepackage{upgreek}

\usepackage{lscape}
\usepackage{color}
\usepackage[color,matrix,arrow]{xy}
\usepackage{setspace}
\xyoption{all}

\DeclareMathOperator{\Ext}{Ext}
\DeclareMathOperator{\Hom}{Hom}
\DeclareMathOperator{\pd}{pd}
\DeclareMathOperator{\id}{id}
\DeclareMathOperator{\add}{add}

\DeclareMathOperator{\End}{End}

\usepackage{avant}
\usepackage{amsmath, amssymb, amsfonts, color, tikz, bbm, txfonts, euscript}
\usepackage{amssymb}
\usepackage{enumerate}
\usepackage{bm}
\newtheorem{theorem}{Theorem}[section]

\newtheorem{prop}[theorem]{Proposition}

\theoremstyle{definition}
\newtheorem{mydef}[theorem]{Definition}

\newtheorem{remark}[theorem]{Remark}

\begin{document}

\thispagestyle{empty}

\title{A New Proof Concerning Quasi-Tilted Algebras}
\author{Stephen Zito\thanks{The author was supported by the University of Connecticut-Waterbury}}
        
\maketitle

\begin{abstract}
Let $\Lambda$ be a quasi-tilted algebra.  If $\Lambda$ is representation-finite, it was shown by Happel, Reiten, and Smal{\o} that $\Lambda$ is tilted.  We provide a new, short proof of this result.
\end{abstract}

\section{Introduction}
Tilting theory is one of the main themes in the study of the representation theory of algebras. Tilted algebras were introduced in $\cite{HR}$ by Happel and Ringel.  Namely, given a tilting module over a hereditary algebra, its endomorphism algebra is called a tilted algebra.  Tilted algebras have been the subject of much investigation and several characterizations are known.  One such characterization was obtained by Jaworska, Malicki, and  Skowro$\acute{\text{n}}$ski in $\cite{JMS}$.  They showed an algebra $\Lambda$ is tilted if and only if there exists a sincere module $M$ such that $\text{Hom}_{\Lambda}(M,\tau_{\Lambda}X)=0$ or $\text{Hom}_{\Lambda}(X,M)=0$ for every module $X\in\mathop{\text{ind}}\Lambda$.
\par  
Quasi-tilted algebras were introduced in $\cite{HRS}$ by Happel, Reiten, and Smal{\o} to give a general tilting theory for abelian categories.  Namely, a quasi-tilted algebra is an endomorphism algebra of a tilting object in a hereditary abelian category.  This definition is generally difficult to handle but the following result, also proved in $\cite{HRS}$, provides 
a useful characterization.  An algebra $\Lambda$ is quasi-tilted if and only if $\text{gl.dim}\leq2$ and, for each $X\in\mathop{\text{ind}}\Lambda$, $\pd_{\Lambda}X\leq1$ or $\id_{\Lambda}X\leq1$.
 Happel, Reiten, and Smal{\o} also proved a representation-finite quasi-tilted algebra is in fact tilted.  The purpose of this paper is to provide a new, short proof of this result using the aforementioned theorem of Jaworska, Malicki, and  Skowro$\acute{\text{n}}$ski.
 \par
 We set the notation for the remainder of this paper. All algebras are assumed to be finite dimensional over an algebraically closed field $k$.  If $\Lambda$ is a $k$-algebra then denote by $\mathop{\text{mod}}\Lambda$ the category of finitely generated right $\Lambda$-modules and by $\mathop{\text{ind}}\Lambda$ a set of representatives of each isomorphism class of indecomposable right $\Lambda$-modules.  Given $M\in\mathop{\text{mod}}\Lambda$, the projective dimension of $M$ in $\mathop{\text{mod}}\Lambda$ is denoted by $\pd_{\Lambda}M$ and its injective dimension by $\id_{\Lambda}M$.  We denote by $\add M$ the smallest additive full subcategory of $\mathop{\text{mod}}\Lambda$ containing $M$, that is, the full subcategory of $\mathop{\text{mod}}\Lambda$ whose objects are the direct sums of direct summands of the module $M$.  We let $\tau_{\Lambda}$ and $\tau^{-1}_{\Lambda}$ be the Auslander-Reiten translations in $\mathop{\text{mod}}\Lambda$.  $D$ will denote the standard duality functor $\Hom_k(-,k)$ and gl.dim will stand for the global dimension of an alegbra.

\subsection{Tilting Modules}
  We begin with the definition of tilting modules.
   \begin{mydef}
   \label{Tilting}
    Let $\Lambda$ be an algebra.  A $\Lambda$-module $T$ is a $\emph{partial tilting module}$ if the following two conditions are satisfied: 
   \begin{enumerate}
   \item[($\text{1}$)] $\pd_{\Lambda}T\leq1$.
   \item[($\text{2}$)] $\Ext_{\Lambda}^1(T,T)=0$.
   \end{enumerate}
   A partial tilting module $T$ is called a $\emph{tilting module}$ if it also satisfies:
   \begin{enumerate}
   \item[($\text{3}$)] There exists a short exact sequence $0\rightarrow \Lambda\rightarrow T'\rightarrow T''\rightarrow 0$ in $\mathop{\text{mod}}\Lambda$ with $T'$ and $T''$ $\in \add T$.
   \end{enumerate}
   \end{mydef}
 Tilting modules induce torsion pairs in a natural way.  We consider the restriction to a subcategory $\mathcal{C}$ of a functor $F$ defined originally on a module category, and we denote it by $F|_{\mathcal{C}}$   \begin{mydef} A pair of full subcategories $(\mathcal{T},\mathcal{F})$ of $\mathop{\text{mod}}\Lambda$ is called a $\emph{torsion pair}$ if the following conditions are satisfied:
   \begin{enumerate}
   \item[(a)] $\text{Hom}_{\Lambda}(M,N)=0$ for all $M\in\mathcal{T}$, $N\in\mathcal{F}.$
   \item[(b)] $\text{Hom}_{\Lambda}(M,-)|_\mathcal{F}=0$ implies $M\in\mathcal{T}.$
   \item[(c)] $\text{Hom}_{\Lambda}(-,N)|_\mathcal{T}=0$ implies $N\in\mathcal{F}.$
   \end{enumerate}
   \end{mydef}
   Consider the following full subcategories of $\mathop{\text{mod}}\Lambda$ where $T$ is a tilting $\Lambda$-module.
 \[
 \mathcal{T}(T)=\{M\in\mathop{\text{mod}}\Lambda~|~ \text{Ext}_{\Lambda}^{1}(T,M)=0\}
 \]
 \[
 \mathcal{F}(T)=\{M\in\mathop{\text{mod}}\Lambda~|~\text{Hom}_{\Lambda}(T,M)=0\}
 \]

 Then $(\mathcal{T}(T),\mathcal{F}(T))$ is a torsion pair in $\mathop{\text{mod}}\Lambda$ called the $\it{induced~torsion~pair}$ of $T$.  We refer the reader to $\cite{ASS}$ for more details.
 \par
 We say a torsion pair $(\mathcal{T},\mathcal{F})$
 is $\it{split}$ if every indecomposable $\Lambda$-module belongs to either $\mathcal{T}$ or $\mathcal{F}$.  We have the following characterization of split torsion pairs.  
  \begin{prop}$\emph{\cite[VI,~Proposition~1.7]{ASS}}$
 \label{split}
 Let $(\mathcal{T},\mathcal{F})$ be a torsion pair in $\mathop{\emph{mod}}\Lambda$.  The following are equivalent:
 \begin{enumerate}
 \item[$\emph{(a)}$] $(\mathcal{T},\mathcal{F})$ is split.
 \item[$\emph{(b)}$] If $M\in\mathcal{T}$, then $\tau_{\Lambda}^{-1}M\in\mathcal{T}$.
 \item[$\emph{(c)}$] If $N\in\mathcal{F}$, then $\tau_{\Lambda}N\in\mathcal{F}$.
 \end{enumerate}
 \end{prop} 
We end this section with a result concerning which torsion pairs are induced by a tilting module in the case $\Lambda$ is representation-finite.
\begin{prop}$\emph{\cite[VI,~Corollary~6.6]{ASS}}$ 
 \label{repfinite}
 Let $\Lambda$ be a representation-finite algebra and $(\mathcal{T},\mathcal{F})$ be a torsion pair in $\mathop{\emph{mod}}\Lambda$.  Then there exists a tilting module $T$ such that $\mathcal{T}=\mathcal{T}(T)$ and $\mathcal{F}=\mathcal{F}(T)$ if and only if $\mathcal{T}$ contains the injectives.
 \end{prop}

 \subsection{Tilted Algebras}

 We now state the definition of a tilted algebra.
 \begin{mydef} Let $A$ be a hereditary algebra with $T$ a tilting $A$-module.  Then the algebra $\Lambda=\End_AT$ is called a $\emph{tilted algebra}$.
 \end{mydef}

There exists many characterizations of tilted algebras. One such characterization of tilted algebras was established in $\cite{JMS}$.  We recall that a $\Lambda$-module $M$ is $\text{sincere}$ if $\text{Hom}_{\Lambda}(\Lambda,M)\neq 0$ for every indecomposable summand of $\Lambda$.
\begin{theorem}$\emph{\cite[Theorem~1]{JMS}}$
\label{Sincere}
An algebra $\Lambda$ is tilted if and only if there exists a sincere module $M$ in $\mathop{\emph{mod}}\Lambda$ such that, for any indecomposable module $X$ in $\mathop{\emph{mod}}\Lambda$,  $\emph{Hom}_{\Lambda}(X,M)=0$ or $\emph{Hom}_{\Lambda}(M,\tau_{\Lambda}X)=0$.
\end{theorem}

\subsection{Quasi-Tilted Algebras}
Quasi-tilted algebras were introduced by Happel, Reiten, and Smal{\o} in $\cite{HRS}$.  The original definition can be difficult to use in practice but Happel, Reiten, and Smal{\o} also proved a useful homological characterization.  Given $X,Y\in\mathop{\text{ind}}\Lambda$, we denote $X\leadsto Y$ in case there exists a chain of nonzero morphisms
\[ 
X=X_0\xrightarrow{f_1}X_1\xrightarrow{f_2}\cdots X_{t-1}\xrightarrow{f_t} X_t=Y
\]
with $t\geq0$, between indecomposable modules.  In this case we say $X$ is a predecessor of $Y$ and $Y$ is a successor of $X$.  Observe that each indecomposable module is a predecessor and successor of itself.  Define the following subcategories of $\mathop{\text{ind}}\Lambda$:
\[
\mathcal{L}_{\Lambda}=\{X\in\mathop{\text{ind}}\Lambda: \text{if}~Y\leadsto X, \text{then} \pd_{\Lambda}Y\leq1\}
\]
\[
\mathcal{R}_{\Lambda}=\{X\in\mathop{\text{ind}}\Lambda: \text{if}~X\leadsto Y, \text{then} \id_{\Lambda}Y\leq1\}
\]
It follows from the definitions that $\mathcal{L}_{\Lambda}$ is closed under predecessors and $\mathcal{R}_{\Lambda}$ is closed under successors.

\begin{theorem}$\emph{\cite{HRS}}$
\label{HRS1}
Let $\Lambda$ be an algebra.  The following are equivalent.
\begin{enumerate}
\item[$\emph{(a)}$] $\Lambda$ is quasi-tilted.
\item[$\emph{(b)}$] $\emph{gl.dim}\leq2$ and, for each $X\in\mathop{\emph{ind}}\Lambda$, $\pd_{\Lambda}X\leq1$ or $\id_{\Lambda}X\leq1$.

\item[$\emph{(c)}$] $\mathcal{L}_{\Lambda}$ contains all indecomposable projective $\Lambda$-modules.
\item[$\emph{(d)}$] $\mathcal{R}_{\Lambda}$ contains all indecomposable injective $\Lambda$-modules.
\end{enumerate}
\end{theorem}
\begin{remark}
\label{remark}
It was also shown in $\cite{HRS}$ that $\mathop{\text{ind}}\Lambda=\mathcal{L}_{\Lambda}\cup\mathcal{R}_{\Lambda}$.  Since $\mathcal{L}_{\Lambda}$ is closed under predecessors, $\mathcal{R}_{\Lambda}$ is closed under successors, and $\mathop{\text{ind}}\Lambda=\mathcal{L}_{\Lambda}\cup\mathcal{R}_{\Lambda}$, it is easy to see that $(\mathop{\text{add}}\mathcal{R}_{\Lambda},\mathop{\text{add}}(\mathcal{L}_{\Lambda}\setminus\mathcal{R}_{\Lambda}))$ is a split torsion pair for any quasi-tilted algebra $\Lambda$.
\end{remark}

\subsection{Auslander-Reiten Translations}
In this section, we record several properties of the Auslander-Reiten translations which will be needed in the proof of a preliminary result.  For further reference, see $\cite{ASS}$.
\begin{prop}$\emph{\cite[IV,~Proposition~2.10]{ASS}}$
\label{Basic}
Let $M$ and $N$ be indecomposable $\Lambda$-modules.
\begin{enumerate}
\item[$\emph{(a)}$] $\tau_{\Lambda}M$ is zero if and only if $M$ is projective.
\item[$\emph{(b)}$] $\tau_{\Lambda}^{-1}N$ is zero if and only if $N$ is injective.
\item[$\emph{(c)}$] If $M$ is a nonprojective module, then $\tau_{\Lambda}M$ is indecomposable and noninjective and $\tau_{\Lambda}^{-1}\tau_{\Lambda}M\cong M$.
\item[$\emph{(d)}$] If $N$ is a noninjective module, then $\tau_{\Lambda}^{-1}N$ is indecomposable and nonprojective and $\tau_{\Lambda}\tau_{\Lambda}^{-1}N\cong N$.
\end{enumerate}
\end{prop}

\begin{prop}$\emph{\cite[IV,~Corollary~2.14]{ASS}}$
\label{AR Corollary}
Let $M$ and $N$ be two $\Lambda$-modules.
\begin{enumerate}
 \item[$\emph{(a)}$] If $\pd_{\Lambda}M\leq1$ and $N$ is arbitrary, then $\emph{Ext}_{\Lambda}^1(M,N)\cong D\emph{Hom}_{\Lambda}(N,\tau_{\Lambda}M)$.
 \item[$\emph{(b)}$] If $\id_{\Lambda}N\leq1$ and $M$ is arbitrary, then $\emph{Ext}_{\Lambda}^1(M,N)\cong D\emph{Hom}_{\Lambda}(\tau_{\Lambda}^{-1}N,M)$.
 \end{enumerate}
\end{prop}

\section{Main Result}
  We begin with a preliminary result.  We provide a sufficient condition for an algebra to be tilted.
\begin{prop}
\label{prop1}
Let $\Lambda$ be an algebra.  If there exists a tilting module $T$ in $\mathop{\emph{mod}}\Lambda$ such that the induced torsion pair $(\mathcal{T}(T),\mathcal{F}(T))$ splits and $\id_{\Lambda}X\leq1$ for every $X\in\mathcal{T}(T)$, then $\Lambda$ is tilted.
\end{prop}
\begin{proof}
We need to show there exists a sincere module $M$ in $\mathop{\text{mod}}\Lambda$ such that, for any $X\in\text{ind}\Lambda$, $\text{Hom}_{\Lambda}(X,M)=0$ or $\text{Hom}_{\Lambda}(M,\tau_{\Lambda}X)=0$.  Theorem $\ref{Sincere}$ will then imply $\Lambda$ is tilted.  Since $T$ is a tilting module, it follows easily from definition $\ref{Tilting}$ (3) that $T$ is sincere.  Let $X\in\mathop{\text{ind}}\Lambda$.  Since $(\mathcal{T}(T),\mathcal{F}(T))$ splits, $X\in\mathcal{T}(T)$ or $X\in\mathcal{F}(T)$.  If $X\in\mathcal{F}(T)$, then $\tau_{\Lambda}X\in\mathcal{F}(T)$ by proposition $\ref{split}$ (c).  Thus, $\Hom_{\Lambda}(T,\tau_{\Lambda}X)=0$ and we are done.  If $X\in\mathcal{T}(T)$, we have two cases to consider: $\tau_{\Lambda}X\in\mathcal{T}(T)$ or $\tau_{\Lambda}X\in\mathcal{F}(T)$.  If $\tau_{\Lambda}X\in\mathcal{F}(T)$, we are done.  If $\tau_{\Lambda}X\in\mathcal{T}(T)$, then $\id_{\Lambda}(\tau_{\Lambda}X)\leq1$ by assumption.  If $X$ is projective, then $\tau_{\Lambda}X=0$ by proposition $\ref{Basic}$ (a) and $\Hom_{\Lambda}(T,\tau_{\Lambda}X)=0$.  If $X$ is nonprojective, then $\tau_{\Lambda}^{-1}\tau_{\Lambda}X\cong X$ by proposition $\ref{Basic}$ (c) and proposition $\ref{AR Corollary}$ (b) implies 
\[
\Ext_{\Lambda}^1(T,\tau_{\Lambda}X)\cong D\Hom_{\Lambda}(\tau_{\Lambda}^{-1}\tau_{\Lambda}X,T)\cong D\Hom_{\Lambda}(X,T)=0.
\]
We conclude by theorem $\ref{Sincere}$ that $\Lambda$ is  a tilted algebra.

\end{proof}
We are now ready for the main result.
\begin{theorem}$\emph{\cite[Corollary~2.3.6]{HRS}}$
If $\Lambda$ is a representation-finite quasi-tilted algebra, then $\Lambda$ is tilted.
\begin{proof}

Remark $\ref{remark}$ shows we have the split torsion pair $(\mathop{\text{add}}\mathcal{R}_{\Lambda},\mathop{\text{add}}(\mathcal{L}_{\Lambda}\setminus\mathcal{R}_{\Lambda}))$ in $\mathop{\text{mod}}\Lambda$.  By Theorem $\ref{HRS1}$ (d), we know $\mathcal{R}_{\Lambda}$ contains all the injective $\Lambda$-modules.  Since $\Lambda$ is representation-finite, proposition $\ref{repfinite}$ guarantees the existence of a tilting module $T$ such that  $(\mathop{\text{add}}\mathcal{R}_{\Lambda},\mathop{\text{add}}(\mathcal{L}_{\Lambda}\setminus\mathcal{R}_{\Lambda}))=(\mathcal{T}(T),\mathcal{F}(T))$.  We also know $\id_{\Lambda}X\leq1$ for every $X\in\mathcal{R}_{\Lambda}$ by the definition of $\mathcal{R}_{\Lambda}$ and the fact that every indecomposable module $X$ is a successor of itself.  We conclude by proposition $\ref{prop1}$ that $\Lambda$ is tilted.

\end{proof}
\end{theorem}

\noindent Mathematics Faculty, University of Connecticut-Waterbury, Waterbury, CT 06702, USA
\it{E-mail address}: \bf{stephen.zito@uconn.edu}

\end{document}